\pdfoutput=1
\documentclass[final]{scrartcl}

\usepackage[english]{babel}

\usepackage[T1]{fontenc}
\usepackage[utf8]{inputenc}

\usepackage[final, 	babel 	]{microtype}

\usepackage{amsfonts} \usepackage{amssymb}  \usepackage{mathtools} \mathtoolsset{showonlyrefs}

\usepackage{amsthm} \theoremstyle{plain}
\newtheorem{thm}[equation]{Theorem}
\newtheorem{prop}[equation]{Proposition}

\newtheorem*{thm*}{Theorem}
\newtheorem*{prop*}{Proposition}

\newtheorem*{principle*}{Principle}

\theoremstyle{definition}

\newtheorem{defn}[equation]{Definition}

\newtheorem*{cor*}{Corollary}
\newtheorem*{lemma*}{Lemma}
\newtheorem*{defn*}{Definition}

\theoremstyle{remark}
\newtheorem{rem}[equation]{Remark}
\newtheorem{ex}[equation]{Example}

\newtheorem*{rem*}{Remark}
\newtheorem*{ex*}{Example} \numberwithin{equation}{section}

\usepackage{tikz} \usetikzlibrary{arrows,matrix,positioning}
\tikzset{pf/.style={>=stealth,->,font=\scriptsize},
	surj/.style={->>},
	inj/.style={right hook->},
	bij/.style={above,sloped,inner sep=0.5pt},
	gl/.style={-,double},
	mat/.style={matrix of math nodes, row sep=2.5em, column sep = 2.5em, text height=1.5ex, text depth=0.25ex},
	dr/.style={matrix of math nodes, row sep=2.5em, column sep = 1.25em, text height=1.5ex, text depth=0.25ex},
	seq/.style={matrix of math nodes, row sep=2em, column sep = 2em, text height=1.5ex, text depth=0.25ex}}
\newenvironment{diag*}{\[\begin{tikzpicture}}{\end{tikzpicture}\]\ignorespacesafterend}
\newenvironment{diag}{\begin{equation}\begin{tikzpicture}[baseline=(current  bounding  box.center)]}{\end{tikzpicture}\end{equation}\ignorespacesafterend}

\usepackage{csquotes} \usepackage[backend=biber,
	style=numeric, 	firstinits=true,
	hyperref=auto, 	isbn=false,
	urldate=comp
	]{biblatex}
\title{Super Riemann Surfaces and the Super Conformal Action Functional}
\author{
	\texorpdfstring{
		Enno Keßler\thanks{kessler@mis.mpg.de}
	}
	{Enno Keßler}
}
\publishers{Max-Planck-Institut für Mathematik in den Naturwissenschaften}

\date{}

\DeclareMathOperator{\Ber}{Ber}

\DeclareMathOperator{\Div}{div}
\DeclareMathOperator{\Diff}{Diff}

\DeclareMathOperator{\End}{End}

\DeclareMathOperator{\GL}{GL}

\DeclareMathOperator{\OGL}{O}

\DeclareMathOperator{\SDiff}{SDiff}

\DeclareMathOperator{\UGL}{U}

\newcommand{\Top}[1]{{\|#1\|}}
\newcommand{\Smooth}[1]{{|#1|}}

\renewcommand{\d}{\mathop{}\!d}

\newcommand{\Dirac}{D \hspace{-2.7mm}\slash\hspace{0.75mm}}

\newcommand{\Lie}{L}

\newcommand{\pt}{pt}

\newcommand{\bbC}{\mathbb{C}}

\newcommand{\bbR}{\mathbb{R}}
\newcommand{\bbZ}{\mathbb{Z}}

\newcommand{\cD}{\mathcal{D}}

\newcommand{\cO}{\mathcal{O}}

\newcommand{\ModuliSp}{\mathcal{M}}
\newcommand{\Teichm}{\mathcal{T}}
\newcommand{\STeichm}{\mathcal{ST}}

\DeclareMathOperator{\SUSY}{SUSY}

\DeclareMathOperator{\susy}{susy}

\begin{document}
\maketitle
\begin{abstract}
	Riemann surfaces are two-dimensional manifolds with a conformal class of metrics.
	It is well known that the harmonic action functional and harmonic maps are tools to study the moduli space of Riemann surfaces.
	Super Riemann surfaces are an analogue of Riemann surfaces in the world of super geometry.
	After a short introduction to super differential geometry we will compare Riemann surfaces and super Riemann surfaces.
	We will see that super Riemann surfaces can be viewed as Riemann surfaces with an additional field, the gravitino.
	An extension of the harmonic action functional to super  Riemann surfaces is presented and applications to the moduli space of super Riemann surfaces are considered.
\end{abstract}

\addsec{Introduction}
The theory of Riemann surfaces is a very old and very interesting topic.
Since the end of the 19th century Riemann surfaces have been explored with different approaches from different areas of mathematics ranging from algebraic geometry to analysis.
In particular the description of Riemann surfaces in terms of conformal classes of metrics and the Teichmüller theory has an interesting connection to harmonic maps and the harmonic action functional (see, for example,~\cite{JJ-CRS}).
Harmonic maps from Riemann surfaces are particular non-linear sigma models.

In contrary, Super Riemann surfaces are a rather new topic.
They appeared in the context of super gravity and super string theory around 1985, see \cites{F-NSTTDCFT}{M-CDSTDSMSC}.
Super Riemann surfaces have been formalized using the language of super geometry (see e.g.~\cite{L-ITS}), an extension of differential or algebraic geometry.
Super Riemann surfaces are particular complex super manifolds of complex dimension \(1|1\).
Even though they possess one even and one odd dimension they are said to behave in certain regards as if they were one-dimensional.
Different approaches from the theory of Riemann surfaces have been “superized”, such as uniformization (\cite{CR-SRSUTT}) and universal deformation spaces (\cite{LBR-MSRS}).

The approach to super Riemann surfaces via super conformality and a super harmonic action functional is very interesting for physics, as it appears in a super symmetric non-linear sigma model.
In \cites{DZ-CASS}{BdVH-LSRIASS} it was proposed to consider a super symmetric extension of the harmonic action functional \(A(\varphi, g)\) where both the metric \(g\) and the field \(\varphi\) get a super partner, \(\psi\) and \(\chi\) respectively.
This particular super symmetric non-linear sigma model is relevant for string theory and super gravity.

It was conjectured that a super Riemann surface \(M\) can be described by a metric \(g\) and the super partner of the metric, called gravitino \(\chi\) on a two dimensional manifold \(\Smooth{M}\).
The action of the non-linear super symmetric sigma model would then be an integral over the super manifold \(M\), resembling the harmonic action functional on \(\Smooth{M}\).
Mathematically this leads to a different approach to the moduli space of super Riemann surfaces.
In a talk at the conference “Quantum Mathematical Physics” in fall 2014 I presented my research to make precise the relation between the super symmetric non-linear sigma model and super Riemann surfaces.
The present paper is a written up version of that talk.

The first chapter gives a brief introduction to the necessary parts of super geometry.
We focus mainly on the local theory, that is the building block \(\bbR^{m|n}\).
Motivation is given by a toy example.

In the second section we will see how super Riemann surfaces can be reduced to Riemann surfaces with an additional gravitino field.
A possible super Teichmüller theory is discussed.

In the third section, the extension of the classical harmonic action functional to super Riemann surfaces is given.
Using the results of the second section it is possible to formulate the super harmonic action functional as an integral over a two-dimensional manifold.
Symmetries of the action functional can be explained with the help of the geometry of super Riemann surfaces.
In analogy to the case of Riemann surfaces, it is expected that the super harmonic action functional may help to understand super Teichmüller space.

\subsection*{Acknowledgement}
I want to thank the organizers of the conference “Quantum Mathematical Physics“, Felix Finster, Jürgen Tolksdorf, and Eberhard Zeidler for inviting me to this most inspiring conference and for the opportunity to present my research.

The research leading to these results has received funding from the European Research Council under the European Union's Seventh Framework Programme (FP7/2007--2013) / ERC grant agreement nº 267087.
The research has been carried out at the Max-Planck-Institut für Mathematik in den Naturwissenschaften as a student of the International Max Planck Research School Mathematics in the Sciences.
I am grateful for the support received.
 
\section{Super geometry}

The theory of super manifolds was developed in the 1970s and 1980s in order to provide a geometrical framework for super symmetric theories in high-energy physics.
Already at that time two different approaches were developed.
One approach is to extend the definition of manifolds in terms of charts by replacing the real numbers by certain Grassmann algebras, see e.g.~\cite{R-GTS}.
The other approach is inspired by algebraic geometry.
It puts emphasis on functions rather than on points.
An early overview article for this approach is~\cite{L-ITS}.
It is proven in~\cite{B-TAS} that both approaches coincide.
We will use here the algebraic approach to super manifolds and certain generalizations given below.

\begin{ex}
	We will motivate and illustrate most definitions in this section by help of the following toy model, inspired by~\cite[§1.3]{DF-SS}.
	Let \(\varphi\) and \(\psi\) (classical) fields on \(\bbR\).
	The main motivation for super geometry is to unify \(\varphi\) and \(\psi\) into one object
	\begin{equation}\label{eq:ToyModelPhi}
		\Phi = \varphi + \eta \psi
	\end{equation}
	and to be able to interpret super symmetry, i.e.\ transformations of the following type
	\begin{align}\label{eq:ToyModelSUSY}
		\delta \varphi &= q\psi &
		\delta \psi &= q \partial_x \varphi
	\end{align}
	in a geometric way.
	To this end one needs to extend the geometrical setting from the domain \(\bbR\) to \(\bbR^{1|1}\), where the objects \(\Phi\), \(\eta\) and the super symmetry transformations get a precise meaning.
\end{ex}

Recall that a locally ringed space \(M\) is a pair \((\Top{M}, \cO_M)\), where \(\Top{M}\) is a topological space and \(\cO_M\) a sheaf of rings on \(M\), see~\cite[§0.4]{EGAI}.
Sections of \(\cO_M\) are called functions.
A homomorphism of ringed spaces \(\varphi\colon M\to N\) is a pair \(\varphi = (\Top{\varphi}, \varphi^\#)\) consisting of a homomorphism \(\Top{\varphi}\colon \Top{M}\to \Top {N}\) of the underlying topological spaces and a sheaf homomorphism \(\varphi^\#\colon \cO_N\to \cO_M\) over \(\varphi\).
\begin{defn}
	We denote by \(\bbR^{m|n}\) the ringed space given by the topological space \(\bbR^m\) together with the sheaf of functions
	\begin{equation}
		\cO_{\bbR^{m|n}} = C^\infty(\bbR^m, \bbR)\otimes \Lambda_n
	\end{equation}
	where \(\Lambda_n\) is a real Grassmann algebra in \(n\) generators.
	A super manifold \(M\) is a ringed space which is locally isomorphic to \(\bbR^{m|n}\).
	We say that \(M\) has \(m\) even and \(n\) odd dimensions, or that \(M\) is of dimension \(m|n\).
	A homomorphism of super manifolds \(\varphi\colon M\to N\) is a homomorphism of locally ringed spaces.
\end{defn}
Let \(x^1,\ldots, x^m\) be the standard coordinate functions on \(\bbR^m\) and \(\eta^1, \ldots, \eta^n\) generators of the Grassmann algebra \(\Lambda_n\).
We call the tuple \((X^A) = (x^a, \eta^\alpha)\) of functions on \(\bbR^{m|n}\) coordinates of \(\bbR^{m|n}\).
Any function \(f\in\cO_{\bbR^{m|n}}\) can be expressed as a finite expansion in the odd coordinates \(\eta^\alpha\):
\begin{equation}
	f = \sum_{\underline{\gamma}} \eta^{\underline{\gamma}} f_{\underline{\gamma}}(x) = f_0 + \eta^\alpha f_\alpha + \dots
\end{equation}
Here the summation runs over all odd multiindices \(\underline{\gamma}\).
The functions \(f_{\underline{\gamma}}\) are ordinary functions on \(\bbR^m\).

Notice that \(\cO_{\bbR^{m|n}}\) inherits a \(\bbZ_2\) grading from the Grassmann algebra \(\Lambda_n\).
We will call elements of \(\cO_{\bbR^{m|n}}\) of parity \(0\) even and elements of parity \(1\) odd.
We use here and in the following the convention, that small roman letters are used for even objects, small greek letters for odd objects and capital letters for even and odd objects together.

In contrast to the theory of manifolds, not every function \(f\in\cO_{\bbR^{m|n}}\) can be seen as a map \(\bbR^{m|n}\to \bbR\).
This is a consequence of the graduation of the structure sheaf \(\cO_{\bbR^{n|m}}\).
By~\cite[Theorem 2.17]{L-ITS}, maps between super domains \(U\subseteq\bbR^{m|n}\) and \(V\subseteq\bbR^{p|q}\) can be given in terms of coordinates.

\begin{ex}\label{ex:MapToR}
	A first possible interpretation for Equation~\eqref{eq:ToyModelPhi} would be that
	\begin{equation}
		\Phi = \varphi + \eta\psi
	\end{equation}
	is a function on \(\bbR^{1|1}\) with coordinates \((x,\eta)\).
	This would however restrict \(\varphi\) and \(\psi\) to be smooth functions on \(\bbR\).
	Even though it looks like a drawback at first sight, the correct way is to consider maps \(\Phi\colon \bbR^{1|1}\to \bbR\).
	Let \(r\) be a coordinate function on \(\bbR\).
	The map \(\Phi\) is then completely determined by the pullback \(\Phi^\# r\) which is an even function on \(\bbR^{1|1}\) because the ring homomorphisms \(\Phi^\#\) preserve automatically the \(\bbZ_2\)-parity of the functions:
	\begin{equation}
		\Phi^\# r = \varphi(x) + \eta \psi(x)
	\end{equation}
	However, if \(\Phi^\# r\) is even the function \(\psi(x)\) has to be zero.
	For the applications we have in mind \(\psi(x)\) is certainly expected to be non-zero.
	Therefore we need to consider a family of maps \(\Phi\) parametrized by a super manifold \(B\), i.e a map that makes the following diagram commutative:
	\begin{diag}
		\matrix[dr](m){
			\bbR^{1|1}\times B & & \bbR\times B\\
			& B & \\};
		\path[pf]
			(m-1-1) edge node[auto]{\(\Phi\)} (m-1-3)
				edge node[auto, swap]{\(p_B\)} (m-2-2)
			(m-1-3) edge node[auto]{\(p_B\)} (m-2-2);
	\end{diag}
	Such a map is again completely determined by the pullback \(\Phi^\# r\) which is this time an even function on \(\bbR^{1|1}\times B\):
	\begin{equation}
		\Phi^\# r = \varphi(x) + \eta \psi(x)
	\end{equation}
	Here the coefficients functions \(\varphi(x)\) and \(\psi(x)\) are functions on \(\bbR^{1|0}\times B\).
	We suppress the \(B\)-dependence in the notation.
	As \(\Phi^\# r\) is even, and \(\eta\) is odd, also \(\psi(x)\) has to be odd.
	This is possible if the base \(B\) possesses odd dimensions.
\end{ex}
The Example~\ref{ex:MapToR} motivates the following definition:
\begin{defn}[\cite{L-ITS}]
	A submersion \(p_M\colon M\to B\) of super manifolds is also called a family of super manifolds over \(B\).
	A morphism \(f\) of families of super manifolds from \(p_M\colon M\to B\) to \(p_N\colon N\to B\) is a morphism \(f\colon M\to N\) such that \(p_N\circ f = p_M\).
	Any super manifold is a family over \(\bbR^{0|0}=(\{\pt\}, \bbR)\).
	Any family is locally a projection \(\bbR^{m|n}\times B\to B\). We call \(m|n\) the dimension of the family.
\end{defn}
According to~\cite[Remark 2.6.(v)]{DM-SUSY} it is not necessary to fix \(B\).
However \(B\) is always supposed to be “big enough”, see Example~\ref{ex:MapToR}.
Henceforth, all super manifolds and maps of super manifolds are implicitly to be understood as families of super manifolds and morphisms of families of super manifolds.
In particular, also \(\bbR^{m|n}\) is to be understood as the trivial family \(\bbR^{m|n}\times B\).

\begin{rem}
	Another quite popular approach to super manifolds is to use the functor of points.
	Full discussions of this approach can be found in~\cite{S-GAASTS}.
	An advantage of this approach is that one can treat infinite dimensional super manifolds.
	Infinite dimensional manifolds can not be treated in the ringed-space approach.
	However non-trivial families of super manifolds are usually not in the scope of the functor of points approach.
	We will see in the next chapter, that we need non-trivial families of super manifolds for the study of moduli spaces.
\end{rem}

It is possible to extend a large part of differential geometry to super manifolds, see e.g.\ \cites{L-ITS}{DM-SUSY}{CCF-MFS}.
In particular there are appropriate definitions of vector bundles, tangent bundles, Lie groups and principle bundles.

\begin{ex}\label{ex:SVect}
	Let \((x^a, \eta^\alpha)\) be coordinates for \(\bbR^{m|n}\).
	Any vector field \(V\) on \(\bbR^{m|n}\) are \(\cO_{\bbR^{m|n}}\)-linear combination of the partial derivatives in coordinate directions:
	\begin{equation}
		V = V^a \partial_{x^a} + V^\alpha \partial_{\eta^\alpha}
	\end{equation}
	A particular vector field on \(\bbR^{1|1}\) is given by the even vector field \(Q=q\left(\partial_\eta - \eta\partial_x\right)\).
	It acts on the function \(\Phi^\# r\) by
	\begin{equation}
		Q \Phi^\# r = q\left(\partial_\eta - \eta \partial_x\right)\left(\varphi(x) + \eta\psi(x)\right) = q\psi(x) + \eta q\partial_x\varphi
	\end{equation}
	The coefficients of \(Q\Phi^\# r\) reproduce the super symmetry transformations from Equation~\eqref{eq:ToyModelSUSY}.
	Consequently, the infinitesimal super diffeomorphism given by the vector field \(Q\) can be identified with the super symmetry transformations~\eqref{eq:ToyModelSUSY}.
\end{ex}

In order to study the relation between super manifolds and ordinary manifolds, we need the concept of an underlying even manifold.
\begin{defn}[{\cite{JKT-SRSMG}}]
	Let \(M=(\Top{M},\cO_M)\) be a family of super manifolds of dimension \(m|n\) over \(B\).
	A family of super manifolds \(\Smooth{M}=(\Top{M},\cO_{\Smooth{M}})\) of dimension \(m|0\) together with an embedding of families of super manifolds \(i\colon \Smooth{M}\to M\) that is the identity on the underlying topological space is called an underlying even manifold.
\end{defn}
In~\cite{JKT-SRSMG} we have shown that such underlying even manifolds always exist.
They are unique, however, only if the odd dimension of \(B\) is zero.

\begin{ex}\label{ex:ToyModelFields}
	Remember that we have defined the fields \(\varphi(x)\) and \(\psi(x)\) as coefficients in the coordinate expansion of
	\begin{equation}
		\Phi^\# r = \varphi(x) + \eta \psi(x)
	\end{equation}
	with respect to fixed coordinates \((x,\eta)\), see Example~\ref{ex:MapToR}.
	This definition is clearly coordinate dependent.
	A general coordinate change on \(\bbR^{1|1}\) (over \(B\)) is given by
	\begin{align}
	\label{eq:R11CoordChange}
		x &= g_0(\tilde{x}) + \tilde{\eta} g_1(\tilde{x}) &
		\eta &= \gamma_0(\tilde{x}) + \tilde{\eta} \gamma_1(\tilde{x})
	\end{align}
	In the coordinates \((\tilde{x}, \tilde{\eta})\) the map \(\Phi\) is given by
	\begin{equation}
		\Phi^\# r = \varphi(g_0(\tilde{x})) + \gamma_0(\tilde{x})\psi(g_0(\tilde{x})) + \tilde{\eta}\left(\frac{\d \varphi}{\d x}(g_0(\tilde{x})) g_1(\tilde{x}) + \gamma_1(\tilde{x}) \psi(g_0(\tilde{x})) \right)
	\end{equation}
	With the help of a chosen embedding \(i\colon \bbR^{1|0}\times B \to \bbR^{1|1}\times B\) we are able to give a coordinate independent definition of \(\varphi\) and \(\psi\).
	Let \(y\) be the standard coordinate on \(\bbR^{1|0}\) and \((x,\eta)\) the standard coordinates on \(\bbR^{1|1}\).
	Any embedding \(i\) can be expressed in those coordinates as:
	\begin{align}
		i^\# x &= y & i^\#\eta = \xi
	\end{align}
	for some odd function \(\xi\) in \(\cO_{\bbR^{1|0}\times B}\).
	At this point it is obvious why the embedding is unique if \(B=\bbR^{0|0}\).

	The automorphism of \(\bbR^{1|1}\times B\) given by
	\begin{align}
		\tilde{x} &= x & \tilde{\eta} &= \eta - \xi
	\end{align}
	yields \(i^\#\tilde{\eta} = 0\).

	Define the field \(\varphi = \Phi\circ i\colon \bbR^{1|0}\to \bbR\).
	One can assume without loss of generality that the embedding \(i\) is given by \(i^\#\eta = 0\).
	Then the degree zero coefficient of \(\Phi^\#r\) coincides with \(\varphi\):
	\begin{equation}
		\Phi^\# r = \varphi(x) + \eta f_1(x)
	\end{equation}
	Note that the choice of \(i\) fixes only one component field.
	Any coordinate change
	\begin{align}
		x &= \tilde{x} + \tilde{\eta} g_1(\tilde{x}) &
		\eta &= \tilde{x} + \tilde{\eta} \gamma_1(\tilde{x})
	\end{align}
	preserves \(\varphi(x)\), but not \(\psi\).
	Any given embedding \(i\) splits all super diffeomorphisms of \(\bbR^{1|1}\) (Equation~\eqref{eq:ToyModelPhi}) into diffeomorphisms of \(\bbR^{1|0}\) (given by \(g_0\)), diffeomorphisms of \(\bbR^{1|1}\) that preserve \(i\) (given by \(g_1\) and \(\gamma_1\)), and diffeomorphisms of \(\bbR^{1|1}\) that change \(i\) (given by \(\gamma_0\)).

	We define the second component field \(\psi\) with the help of the vector field \(D=\partial_\eta + \eta\partial_x\).
	The vector field \(D\) is characterized by the property that it commutes with the super symmetry vector field \(Q\) given in~\ref{ex:SVect}.
	The definition \(\psi = i^*D\Phi\) then assures that \(\psi\) is the super partner to \(\varphi\), because the action of the vector field \(Q\) on the component fields is given by
	\begin{align}
		\delta \varphi &= i^*Q\Phi = q\psi &
		\delta \psi &= i^* Q D \Phi = q \partial_x \varphi
	\end{align}
	The definition of \(\psi\) given here shows that \(\psi\) is a section of \(\varphi^*T\bbR\) and is independent of the chosen coordinates.
	The vector field \(D\) encountered here is crucial for the definition of super Riemann surfaces in the next chapter.
	The particular structure of super Riemann surfaces will then also assure that, contrary to our toy model here, \(\psi\) is a spinor.
\end{ex}

Integrals over a super manifold \(M\) can be reduced to integrals over \(\Smooth{M}\) via an embedding \(i\colon \Smooth{M}\to M\).
Integration is defined for sections of \(\Ber T^\vee M\), a generalization of the determinant line bundle.
Integration is given in local coordinates \((x^a,\eta^\alpha)\) such that \(i^\#\eta^\alpha = 0\) by
\begin{equation}\label{eq:IntCoord}
	\int_{\bbR^{m|n}} g(x,\eta)[\d x^1\ldots\d x^m\d\eta^1\ldots\d\eta^n]  = \int_{\bbR^{m|0}} g_{top}(x) \d x^1\ldots\d x^m
\end{equation}
where \(g_{top}\) is the coefficient of \(\eta^1\cdot\cdots\cdot\eta^n\) in the coordinate expansion of \(g\).

\begin{ex}
	In our toy model, a super symmetric action for the fields \(\varphi\) and \(\psi\) is given by
	\begin{equation}
		A(\varphi, \psi) = \frac12 \int_\bbR \varphi'^2 + \psi \psi' \d{x}
	\end{equation}
	The action \(A(\varphi, \psi)\) can be formulated in terms of super symmetry via an integral over \(\bbR^{1|1}\) where the integrand depends on \(\Phi\) as follows:
	\begin{equation}
		A(\varphi, \psi) = A(\Phi) = -\frac12\int_{\bbR^{1|1}} \partial_x \Phi D \Phi [\d{x}\d\eta]
	\end{equation}
	Note that the reduction of the integral over \(\bbR^{1|1}\) to an integral over \(\bbR\) is given with respect to \(i\).
	However the definition of \(A(\Phi)\) does not depend on \(i\).
	Consequently the integral \(A(\Phi)\) has an additional symmetry, the change of embedding \(i\).
	An infinitesimal change of embedding \(i\) is given by the even vector field \(Q = q\left(\partial_\eta - \eta\partial_x\right)\).
	The super symmetry of \(A(\varphi, \psi)\) can thus be interpreted geometrically, in terms of a change of embedding of the underlying even manifold.
\end{ex}
 
\section{Super Riemann surfaces}

Super Riemann surfaces are \(2|2\)-dimensional super manifolds with additional structure.
They appeared in the 1980s in the context of string theory and super gravity.
Early references are \cites{F-NSTTDCFT}{M-CDSTDSMSC}{GN-GSRS}{LBR-MSRS}.
We will see in this section how they can be considered as a generalization of classical Riemann surfaces and give an outlook to a possible super Teichmüller theory.

Let us recall, that there are several different ways to define and study Riemann surfaces.
From the viewpoint of complex geometry, Riemann surfaces are \(1\)-dimensional complex manifolds.

In differential geometry one can describe Riemann surfaces as two-dimensional (real) manifolds with additional geometric structure, given by a conformal class of metrics or an almost complex structure.
Let \(\Smooth{M}\) be a two-dimensional smooth manifold of genus \(p\).
Let furthermore \(g\) and \(\tilde{g}\) be two Riemannian metrics on \(\Smooth{M}\).
Recall that the metrics \(g\) and \(\tilde{g}\) belong to the same conformal class \([g]\) if there is a positive function \(\Lambda\) such that \(g=\Lambda \tilde{g}\).
In two dimensions, a conformal class of metrics together with an orientation induces an almost complex structure \(I\) by
\begin{equation}
	g(IX, Y) = \d vol_g(X,Y)
\end{equation}
for all vector fields \(X\) and \(Y\).
It is also particular to the two-dimensional case that this almost complex structure is always integrable, i.e.\ leads to a complex manifold.

Let \(f\colon \Smooth{M}\to \Smooth{M}\) be a diffeomorphism.
The metric spaces \((\Smooth{M},g)\) and \((\Smooth{M}, f^*g)\) are isometric.
Consequently the resulting Riemann surfaces are isomorphic and isomorphism classes of Riemann surfaces are described by the quotient of conformal classes up to diffeomorphisms:
\begin{equation}
\label{eq:ModuliSpDiffGeo}
	\ModuliSp_p = \left\{\text{conformal classes } [g] \text{ on }\Smooth{M}\right\}/\Diff\Smooth{M}
\end{equation}
Unfortunately the isomorphism classes of Riemann surfaces cannot be endowed with a manifold structure.
However, an infinite cover of this space can be equipped with a manifold structure, the Teichmüller space:
\begin{equation}
\label{eq:Teichm}
	\Teichm_p = \left\{\text{conformal classes } [g] \text{ on }\Smooth{M}\right\}/\Diff_0\Smooth{M}
\end{equation}
Here \(\Diff_0\Smooth{M}\) denotes diffeomorphisms of \(\Smooth{M}\) that are homotopic to the identity.
It is a theorem due to Oswald Teichmüller, that the Teichmüller space \(\Teichm_p\) is isomorphic to \(\bbR^{6p-6}\).

Infinitesimal deformations of a given Riemann surface \((\Smooth{M},g)\) are tangent vectors to the appropriate point in Teichmüller space.
Since Riemann surfaces are described here in terms of Riemannian metrics, infinitesimal deformations of Riemann surfaces are given by infinitesimal changes \(\delta g\) of the metric \(g\).
Any infinitesimal change of the metric \(\delta g\) can be decomposed into infinitesimal conformal rescaling, Lie derivative of \(g\) (infinitesimal diffeomorphism) and “true infinitesimal deformations” of the Riemann surface:
\begin{equation}
\label{eq:InfDecompositionOfMetric}
	\delta g = \lambda g + \Lie_X g + D
\end{equation}
It can be shown that the true infinitesimal deformations \(D\) are holomorphic quadratic differentials, i.e.\ holomorphic sections of \(T^\vee\Smooth{M}\otimes_\bbC T^\vee\Smooth{M}\).

Super Riemann surfaces can also be described and studied with more algebraic or more differential geometric methods.
After a brief look at the algebraic definition of super Riemann surfaces and its consequences we will turn to a more differential geometric treatment of super Riemann surfaces.
We will see that the differential geometric picture allows to describe a super Riemann surface \(M\) in terms of a metric \(g\), a spinor bundle \(S\), and a gravitino field \(\chi\) on an underlying even manifold \(\Smooth{M}\).
This is a precise version of a conjecture to be found in \cites{dHP-GSP}{JJ-GP}.

We use here the  algebraic definition of super Riemann surfaces given in \cites{M-CDSTDSMSC}{LBR-MSRS}.
\begin{defn}
	A super Riemann surface is a \(1|1\)-dimensional complex super manifold \(M\) with a \(0|1\)-dimensional distribution \(\cD\subset TM\) such that the commutator of vector fields induces an isomorphism
	\begin{equation}
		\frac12 [\cdot, \cdot]\colon \cD\otimes_\bbC\cD \to TM/\cD.
	\end{equation}
\end{defn}
\begin{ex}\label{ex:SRS}
	Let \((z,\theta)\) be the standard coordinates on \(\bbC^{1|1}\) and define \(\cD\subset T\bbC^{1|1}\) by \(\cD=\langle\partial_\theta + \theta\partial_z\rangle\). The isomorphism \(\cD\otimes\cD \simeq TM/\cD\) is explicitly given by
\begin{equation}
		[\partial_\theta + \theta\partial_z, \partial_\theta + \theta\partial_z] = 2\partial_z
	\end{equation}
	This example is generic since any super Riemann surface is locally of this form, see~\cite[Lemma 1.2]{LBR-MSRS}.
\end{ex}
The following proposition is an easy consequence of this definition:
\begin{prop}[see e.g. {\cite[Proposition 4.2.2]{S-GAASTS}}]
\label{prop:BijectionSRSSpinCurves}
	There exists a bijection between the set of super Riemann surfaces over \(\bbR^{0|0}\) and the set of pairs \((\Smooth{M}, S)\), where \(S\) is a spinor bundle over the Riemann surface \(\Smooth{M}\), i.e. \(S\otimes_\bbC S = T\Smooth{M}\).
\end{prop}
\begin{proof}
	As indicated in the Example~\ref{ex:SRS}, the super Riemann surface \(M\) can be covered by coordinate charts \((z,\theta)\) such that the holomorphic line bundle \(\cD\) is generated by \(\partial_\theta + \theta\partial_z\).
	Suppose \((z,\theta)\) and \((\tilde{z}, \tilde{\theta})\) are two pairs of such coordinates.
	In the formula for the holomorphic change of coordinates
	\begin{align}\label{eq:BijectionSRSSpinCurvesCoordChange}
		\tilde{z} &= f(z) & \tilde{\theta} &= g(z) \theta
	\end{align}
	the holomorphic functions \(f(z)\) and \(g(z)\) are related by the condition that \(\partial_{\tilde{\theta}} + \tilde{\theta}\partial_{z}\) must be proportional to \(\partial_\theta + \theta\partial_z\).
	One can check that
	\begin{equation}
		\partial_\theta + \theta\partial_z = g(z) \left(\partial_{\tilde{\theta}} + \tilde{\theta} \partial_{\tilde{z}}\right)
	\end{equation}
	if and only if \(f'(z) = {g(z)}^2\).
	As the unique underlying even manifold is given by \(\theta = 0\), the coordinates \(z\) induce a complex structure on \(\Smooth{M}\).
	The functions \(g(z)\) can be used as patching functions for a line bundle \(S\) such that \(S\otimes_\bbC S = T\Smooth{M}\).
	As explained before, a complex structure on a two-dimensional manifold corresponds to a conformal class of metrics \([g]\).
	It can be shown that complex line bundles \(S\) such that \(S\otimes_\bbC S = T\Smooth{M}\) are spinor bundles associated to a spin structure to any metric \(g\) in the conformal class.
\end{proof}

The Proposition~\ref{prop:BijectionSRSSpinCurves} shows that super Riemann surfaces over \(\bbR^{0|0}\) are in one to one correspondence to Riemann surfaces with spinor bundles.
For non-trivial families of super Riemann surfaces \(M\to B\) the proof of Proposition~\ref{prop:BijectionSRSSpinCurves} fails because the change of variables formula~\eqref{eq:BijectionSRSSpinCurvesCoordChange} can get more complicated in the presence of odd dimensions in the base \(B\) (see~\cite{CR-SRSUTT}).
We will see below that the additional information of a spinor valued differential form \(\chi\) is needed to describe non-trivial families of super Riemann surfaces.

It was furthermore shown in \cites{LBR-MSRS}[Theorem 8.4.4]{S-GAASTS} that there is a semi-universal family \(\mathcal{E}\to \STeichm_p\) of super Riemann surfaces of genus \(p\).
That is any family \(M\to B\) of super Riemann surfaces can be obtained in a non-unique way as a pullback of \(\mathcal{E}\) along a map \(B\to \STeichm_p\).
The base manifold \(\STeichm_p\) is a super manifold over \(\bbR^{0|0}\) of real dimension \(6p-6|4p-4\).
Proposition~\ref{prop:BijectionSRSSpinCurves} proves that the points of \(\Smooth{\STeichm_p}\), i.e.\ maps \(\bbR^{0|0}\to \STeichm_p\), are in one to one correspondence to Riemann surfaces with a chosen spinor bundle.
The super structure of \(\STeichm_p\) is encoded in non-trivial families of super Riemann surfaces.
In order to study non-trivial families of super Riemann surfaces we will turn to a more differential geometric description of super Riemann surfaces:

\begin{thm}[\cite{GN-GSRS}]\label{thm:SRSReductionOfStructureGroup}
	A super Riemann surface is a \(2|2\)-dimensional real super manifold with a reduction of the structure group of its frame bundle to
	\begin{equation}
		G = \left\{
			\begin{pmatrix}
				A^2 & B\\
				0 & A \\
			\end{pmatrix}
		\middle| A, B\in\bbC
		\right\} \subset \GL_\bbC(1|1) \subset \GL_\bbR(2|2)
	\end{equation}
	together with suitable integrability conditions.
	Remember that \(\bbC\) is to be understood as the trivial family \(\bbC\times B\).
\end{thm}
Theorem~\ref{thm:SRSReductionOfStructureGroup} is also interesting for the physical motivation of super Riemann surfaces.
It was shown in~\cite{GN-GSRS} that the integrability conditions of Theorem~\ref{thm:SRSReductionOfStructureGroup} are related to the so-called “Torsion-constraints” that can be found in more physics-oriented papers as~\cites{DZ-CASS}{dHP-GSP}.
For a more recent approach via connections on super manifolds and their torsion consult~\cites{L-TCS}{EK-UCSRS}.

The Theorem~\ref{thm:SRSReductionOfStructureGroup} shows that it is impossible to describe the geometry of super Riemann surfaces as a super conformal class of super metrics.
Any orthogonal matrix that is upper triangular would indeed be diagonal.
Consequently \(\OGL(2|2)\nsubseteq G\), i.e.\ the choice of a super metric is not sufficient to determine a super Riemann surfaces.
However there are particular super metrics that are compatible with the structure of a super Riemann surface.
They are given by further reduction to \(\UGL(1)\) as follows
\begin{equation}
	\begin{split}
		\UGL(1)&\to G\\
		U &\mapsto
			\begin{pmatrix}
				U^2 & 0 \\
				0 & U \\
			\end{pmatrix}.
	\end{split}
\end{equation}
A further reduction of the structure group to \(\UGL(1)\) as above leads to a splitting of the following short exact sequence:
\begin{diag}\label{seq:SRSsplit}
	\matrix[seq](m) { 0 & \cD & TM=\cD^\perp\oplus \cD & TM/\cD & 0\\ };
	\path[pf]	(m-1-1) edge (m-1-2)
		(m-1-2) edge (m-1-3)
		(m-1-3) edge (m-1-4)
		(m-1-4) edge (m-1-5)
			edge[bend right=40] node[auto]{\(p\)} (m-1-3);
\end{diag}
Consider now an embedding of an underlying even manifold \(i\colon \Smooth{M}\to M\) for a fixed super Riemann surface \(M\).
Recall that the underlying even manifold \(\Smooth{M}\) of the super manifold \(M\) is a family of super manifolds of relative dimension~\(2|0\) over the base \(B\).
The pullback of the short exact sequence~\eqref{seq:SRSsplit} along an embedding \(i\)
\begin{diag}\label{eq:SRS_pullback_seq}
	\matrix[seq](m) { 0 & S & i^*TM & T\Smooth{M} & 0\\};
	\path[pf]	(m-1-1) edge (m-1-2)
		(m-1-2) edge (m-1-3)
		(m-1-3) edge (m-1-4)
		(m-1-4) edge (m-1-5)
			edge[bend right=40] node[auto]{\(\tilde{p}\)} (m-1-3)
			edge[bend left=40] node[auto]{\(Ti\)} (m-1-3);
\end{diag}
possesses a second splitting given by \(Ti\).
By the identification \(T\Smooth{M}=i^*\cD^\perp\), the tangent bundle of \(\Smooth{M}\) gets equipped with a metric \(g\).
The bundle \(S=i^*\cD\) is a spinor bundle of the metric \(g\) because \(i^*\cD\otimes_\bbC i^*\cD = i^*TM/\cD = T\Smooth{M}\).

The difference of the splittings \(\tilde{p}\) and \(Ti\) is a section of \(T^\vee\Smooth{M}\otimes S\) which we call gravitino \(\chi\).
\begin{equation}\label{eq:DefinitionGravitino}
	\chi(v) = p_S\left(\tilde{p} - Ti\right)v.
\end{equation}
Here \(p_S\colon i^*TM\to S\) is the projector given by the splitting of the short exact sequence by \(\tilde{p}\).

The construction given here associates to any super Riemann surface \(M\) with additional \(\UGL(1)\)-structure a triple \((g, S, \chi)\) that consists of a metric \(g\), a spinor bundle \(S\) and a gravitino field \(\chi\) on the underlying surface \(\Smooth{M}\).
Different choices of \(\UGL(1)\)-structure on the same super Riemann surface lead to metrics and gravitinos which differ from \(g\) and \(\chi\) only by a conformal and super Weyl transformation.
A super Weyl transformation is a transformation of the gravitino given by
\begin{align}
	\chi(v) \mapsto \chi(v) + \gamma(v) t
\end{align}
Here \(t\) is a section of \(S\) and \(\gamma\colon T\Smooth{M}\to \End(S)\) is Clifford multiplication and \(v\) a tangent vector field to \(\Smooth{M}\).

It is surprising that the metric \(g\), the spinor bundle \(S\) and the gravitino \(\chi\) contain full information about the super Riemann surface \(M\).
Indeed, it was shown in~\cite{JKT-SRSMG} that the super Riemann surface \(M\) and the embedding \(i\colon\Smooth{M}\to M\) can be reconstructed from the metric, the spinor bundle and the gravitino.
Thus there is a bijection
\begin{multline}
	\left\{i\colon \Smooth{M}\to M, M \text{ super Riemann surface}\right\} \\
	\longleftrightarrow \left\{ \Smooth{M}, S, g, \chi \right\} / \text{Weyl, SWeyl}.
\end{multline}

The metric \(g\) and the gravitino \(\chi\) do explicitly depend on the embedding~\(i\).
The normal bundle to the embedding \(i\colon\Smooth{M}\to M\) is \(S=i^*\cD\).
Thus an infinitesimal deformation of the embedding \(i\) is given by a section \(q\) of \(S\).
The resulting infinitesimal change of metric and gravitino is given by (c.f.~\cite{JKT-SRSMG}):
\begin{equation}\label{eq:MetricGravitinoSUSY}
	\begin{split}
		\delta f_a &= -2\langle\gamma^b q, \chi(f_a)\rangle f_b \\
		\delta \chi_a&= \nabla_{f_a}^S q = \nabla^{LC}_{f_a} q + \langle \gamma^b\chi_b, \chi_a\rangle \gamma^1\gamma^2q
	\end{split}
\end{equation}
Here the metric is expressed in terms of an orthonormal frame \(f_a\) and the gravitino in components \(\chi_a = \chi(f_a)\).
The spinor covariant derivative \(\nabla^{LC}\) is the Levi-Civita connection lifted to \(S\).
Equations~\eqref{eq:MetricGravitinoSUSY} are known as super symmetry transformations of the metric and gravitino and will be a symmetry of the action functional \(A(\varphi, g, \psi, \chi, F)\) below.
Furthermore one can show that it is possible to choose an embedding \(i\) such that the gravitino vanishes around a given point \(p\in\Smooth{M}\).
If \(M\) is a trivial family of super Riemann surfaces it is possible to choos an embedding \(i\) such that the gravitino vanishes on the whole of \(\Smooth{M}\).

Having a description of super Riemann surfaces in terms of metrics and gravitinos, it is natural to ask for a description of the super moduli space in terms of metrics and gravitinos.
Conjectures about such a super Teichmüller space can be found in the literature, see e.g.~\cites[Equation 3.85]{dHP-GSP}{JJ-GP}.
It is expected that there is a one-to-one correspondence
\begin{multline}
\label{eq:SuperModuliQuotient}
	\left\{M, M \text{ super Riemann surface}\right\}/\SDiff(M) \\
	\longleftrightarrow \\
	\left\{ \Smooth{M}, S, g, \chi\right\} / \text{Weyl, SWeyl,} \Diff(\Smooth{M}), \SUSY
\end{multline}
The group of super symmetry transformations \(\SUSY\) on the right hand side can probably be identified with the change of embedding \(i\).
A precise definition of \(\SUSY\) and the study of the full quotient must be left for further research.
A particularly interesting question is, how the quotient by \(\SUSY\) is related to the nonprojectedness of the super moduli space (see \cite{DW-SMNP}).

However it is possible to study an infinitesimal version of the quotient~\eqref{eq:SuperModuliQuotient}.
That is infinitesimal deformations of super Riemann surfaces can be studied in terms of infinitesimal deformations of metrics and gravitino.
Similar to Equation~\eqref{eq:InfDecompositionOfMetric} it is possible to decompose the infinitesimal deformations
\begin{align}
	\delta g &= \lambda g + \Lie_X g + \susy(q) + D \\
	\delta\chi &= \gamma t + \Lie_X \chi + \susy(q) + \mathfrak{D}
\end{align}
Here \(\susy(q)\) denotes the infinitesimal super symmetry transformations from Equation~\eqref{eq:MetricGravitinoSUSY}.
\(\lambda g\) is the infinitesimal Weyl transformation and \(\gamma t\) the infinitesimal super Weyl transformation.
It is possible to determine the free parameters \(\lambda\), \(t\), \(X\), and \(q\) such that the remaining “true deformations” \(D\) and \(\mathfrak{D}\) are holomorphic sections of \(T^\vee\Smooth{M}\otimes_\bbC T^\vee\Smooth{M}\) and \(S^\vee\otimes_\bbC S^\vee\otimes_\bbC S^\vee\) respectively.
More precisely one has the following:
\begin{thm}[{\cite{JKT-SRSMG}}]
	Let \(M\) be the super Riemann surface given by \(g\), \(S\) and \(\chi\) under the embedding \(i\colon\Smooth{M}\to M\).
	The infinitesimal deformations of \(M\) are given by
	\begin{equation}
		H^0(T^\vee\Smooth{M}\otimes_\bbC T^\vee\Smooth{M})\oplus H^0(S^\vee\otimes_\bbC S^\vee\otimes_\bbC S^\vee)
	\end{equation}
	Here \(H^0\) denotes holomorphic sections.
\end{thm}
This result is well-known (see e.g.~\cite{S-GAASTS} and references therein).
However, the approach outlined here gives a much more geometrical description of the even and odd infinitesimal deformations as infinitesimal deformations of the metric and gravitino, respectively.
 
\section{Action functional}
In this chapter we are investigating a super symmetric extension of the harmonic action functional on Riemann surfaces.
This non-linear super symmetric sigma model can be formulated as an integral over a super Riemann surface and may help, like the harmonic action functional on Riemann surfaces, to understand the moduli space of super Riemann surfaces.

Let us first recall how the harmonic action functional on Riemann surfaces can be used as a tool to study the Teichmüller space.
Details can be found in~\cites{JJ-BS}{JJ-CRS}.
Let \(\varphi\colon\Smooth{M}\to N\) be a smooth map from the Riemann surface \((\Smooth{M}, g)\) to the Riemannian manifold \((N,n)\).
The harmonic action functional as a functional of the metric \(g\) and the map \(\varphi\) is given by
\begin{equation}
	\label{eq:HarmAF}
	A(g, \varphi) = \int_{\Smooth{M}} \|\d\varphi\|_{g^\vee\otimes \varphi^*n}^2 \d vol_g
\end{equation}
The maps \(\varphi\) which are critical points of \(A(g, \varphi)\) are called harmonic maps.

In the case of two dimensional domains, as considered here, the action is conformally invariant, i.e. \(A(\Lambda g, \varphi) = A(g, \varphi)\).
The action functional \(A(g,\varphi)\) can thus be considered as a functional on the conformal class of metrics.
Furthermore it is diffeomorphism invariant, i.e.\ for any diffeomorphism \(f\colon\Smooth{M}\to\Smooth{M}\)
\begin{equation}
	A(f^*g, \varphi\circ f) = A(g, \varphi)
\end{equation}
The harmonic action functional can thus be viewed as a functional on isomorphism classes of Riemann surfaces in the sense of Equation~\eqref{eq:ModuliSpDiffGeo} and also on the Teichmüller space, see Equation~\eqref{eq:Teichm}.

Define the energy-momentum-tensor as the variation of \(A(g, \varphi)\) with respect to the metric \(g\):
\begin{equation}
\label{eq:EnergyMomentumTensor}
	\delta_g A(g, \varphi) = \int_{\Smooth{M}}\delta g \cdot T \d vol_g
\end{equation}
For a harmonic map \(\varphi\) the energy-momentum-tensor is the Noether current that corresponds to diffeomorphism invariance.
Infinitesimal conformal rescalings \(\delta g = \lambda g\)  and Lie-derivatives \(\delta g = \Lie_X g\) must lie in the kernel of the variation \(\delta_g A\).
The vanishing of \(\delta_g A\) on infinitesimal conformal rescalings and Lie-derivatives is sufficient to show that the energy-momentum tensor \(T\) can be identified with a holomorphic quadratic differential.

The preceding facts are particularly interesting in the case where the codomain \((N, n)\) is also a Riemann surface.
Let us assume that the genus of Riemann surfaces \(\Smooth{M}\) and \(N\) is strictly larger than one.
It is possible to assume that the metrics \(g\) and \(n\) have constant curvature \(-1\).
If, furthermore, the Riemann surfaces \((\Smooth{M},g)\) and \((N,n)\) are of the same genus \(p\) there is a unique harmonic map \(\varphi\colon \Smooth{M}\to N\) homotopic to the identity (see~\cite[Corollary 3.10.1]{JJ-CRS}).
The energy-momentum-tensor \(T\) gives a map
\begin{equation}
\label{eq:TeichmMap}
	\Teichm_p \to \Gamma_\Smooth{M} (T^\vee\Smooth{M}\otimes_\bbC T^\vee\Smooth{M})
\end{equation}
sending the Riemann surface \((N,n)\) to the holomorphic quadratic differential \(T\) associated to \(\varphi\).
The Teichmüller theorem (\cite[Thm 4.2.2]{JJ-CRS}) states that the above map~\eqref{eq:TeichmMap} is a diffeomorphism.
The theorem of Riemann-Roch shows that the right-hand side of~\eqref{eq:TeichmMap}, is a finite-dimensional vector space isomorphic to \(\bbC^{3p-3}\simeq\bbR^{6p-6}\).

Summing up, we have seen that the harmonic action functional and harmonic maps help to prove fundamental results in the theory of Teichmüller space.
Furthermore, as is shown in~\cite{JJ-BS}, harmonic maps are also useful to study a quantized version of the harmonic action functional.
We now present the outline of a similar theory in the case of super Riemann surfaces.

Let \(M\) be a super Riemann surface with a fixed \(\UGL(1)\)-structure and local \(\UGL(1)\)-frames \(F_A\).
Let \(\Phi\colon M\to N\) be a map to an arbitrary Riemannian (super) manifold \((N,n)\).
The super symmetric extension of the harmonic action functional~\eqref{eq:HarmAF} is given by
\begin{equation}
\label{eq:SuperAF}
	A(M, \Phi) = \int_{M} \|\left.\d\Phi\right|_\cD\|^2 [\d vol] = \int_M \varepsilon^{\alpha\beta}\langle F_\alpha \Phi, F_\beta \Phi\rangle_{\Phi^*n} [F^1F^2F^3F^4]
\end{equation}
At a first glance this action functional looks just like the harmonic action functional~\eqref{eq:HarmAF}.
However notice that the norm of the differential \(\d\Phi\) restricted to \(\cD\) is used.
This difference to the harmonic action functional is crucial to show that the action functional~\eqref{eq:SuperAF} does only depend on the underlying \(G\)-structure and not on the chosen \(\UGL(1)\)-structure.
This analogue of conformal invariance, together with the super diffeomorphism invariance of \(A(M,\Phi)\) turns \(A(M, \Phi)\) into a functional on the moduli space of super Riemann surfaces as in the left-hand side of Equation~\eqref{eq:SuperModuliQuotient}.

The action functional~\eqref{eq:SuperAF} can be found at different places in the literature, see for example \cites{dHP-GSP}{GN-GSRS}.
In~\cite{GN-GSRS} the super conformal invariance of the action functional \(A(M,\Phi)\) is shown.

The maps \(\Phi\) that are critical with respect to \(A(M,\Phi)\) are described by a differential equation of second order:
\begin{equation}\label{eq:EL}
	0 = \Delta^\cD \Phi = \varepsilon^{\alpha\beta}\nabla_{F_\alpha} F_\beta\Phi + \varepsilon^{\alpha\beta}\left(\Div F_\alpha\right)F_\beta\Phi
\end{equation}
Analytical properties of the \(\cD\)-Laplace operator \(\Delta^\cD\), defined here, still need to be studied.
Remember that a detailed understanding of the analysis of harmonic maps is crucial for the definition of the Teichmüller map~\eqref{eq:TeichmMap}.
In~\cite{JKT-SRSMG} it was shown that the Equation~\eqref{eq:EL} can be used to derive equations of motion for the component fields defined below in Definition~\ref{defn:CompFields}.

Let now \(i\colon\Smooth{M}\to M\) be an underlying even manifold for \(M\).
We have seen in the last section that the super Riemann surface \(M\) can be described in terms of a metric field \(g\) and a gravitino field \(\chi\).
As explained in the first section, every integral over a super manifold can be reduced to an integral over the underlying even manifold.
Let us denote the resulting Lagrangian density on \(\Smooth{M}\) by \(\Smooth{L}\), i.e.:
\begin{equation}
	A(M, \Phi) = \int_{\Smooth{M}} \Smooth{L}
\end{equation}
Of course the Lagrangian density \(\Smooth{L}\) depends not only on \(g\) and \(\chi\), but also on \(\Phi\).
In order to express this dependence in a geometric way, we will now introduce component fields for \(\Phi\).

\begin{defn}
\label{defn:CompFields}
	Let \(\Phi\colon M\to N\) be a morphism and \(i\colon \Smooth{M}\to M\) be an underlying even manifold.
	We call the fields
	\begin{align}
		\varphi\colon \Smooth{M}&\to N & \psi\colon \Smooth{M}&\to S^\vee\otimes \varphi^*TN & F\colon \Smooth{M}&\to \varphi^*TN \\
		\varphi &= \Phi\circ i & \psi &= s^\alpha \otimes i^* F_\alpha\Phi & F &= i^*\Delta^\cD\Phi
	\end{align}
	component fields of \(\Phi\).
	The vectors \(s^\alpha\) form the dual basis to the basis \(s_\alpha = i^*F_\alpha\) of the spinor bundle \(S = i^*\cD\) on \(\Smooth{M}\).
\end{defn}

The component fields are sufficient to fully determine the map \(\Phi\).
There are particular coordinates \((x^a,\eta^\alpha)\) on \(M\) such that in the case \(N=\bbR\) the map \(\Phi\) can be written as
\begin{equation}
	\Phi^\#r = \varphi + \eta^\mu\psi_\mu + \eta^1\eta^2 F
\end{equation}
This expansion is similar to the Example~\ref{ex:ToyModelFields}.
It is an advantage of the geometric definition of the component fields \(\varphi\), \(\psi\) and \(F\), used here, to apply for arbitrary target manifolds \(N\).

With the help of the component maps \(\varphi\), \(\psi\) and \(F\), as well as \(g\) and \(\chi\), the Lagrangian density \(\Smooth{L}\) can be calculated explicitly.
For details on the rather long computations, see the forthcoming thesis~\cite{EK-DR}.
Note that the reduction of the action functional \(A(M,\Phi)\) to the action functional \(A(\varphi, g, \psi, \chi, F)\) given below was claimed in the literature almost 30~years ago (see, for example,~\cite{dHP-GSP}).
This reduction served as a major motivation for the introduction of super Riemann surfaces, as supposedly the action functional \(A(M, \Phi)\) would be easier to study than the component action \(A(\varphi, g, \psi, \chi, F)\).
However, no proof of the reduction could be found in the literature.
\begin{thm}
\label{thm:AF}
	Let \(M\) be a super Riemann surface and \(i\colon \Smooth{M}\to M\) an underlying even manifold.
	We denote by \(g\), \(\chi\), and \(g_S\) respectively the metric, gravitino, and spinor metric on \(\Smooth{M}\) induced by a given \(\UGL(1)\)-structure on \(M\).
	Let \(\Phi\colon M\to N\) be a morphism to a Riemannian super manifold \((N,n)\) and \(\varphi\), \(\psi\), and \(F\) its component fields, as introduced in Definition~\ref{defn:CompFields}.
	One obtains
	\begin{multline}\label{eq:AFRed}
		A(M, \Phi) = A(\varphi, g, \psi, \chi, F) = \int_{\Smooth{M}} \| \d\varphi\|^2_{g^\vee\otimes \varphi^*n} + \langle \psi, \Dirac\psi\rangle_{g_S^\vee\otimes \varphi^*n} - \frac14\langle F, F\rangle_{\varphi^*n} \\
		+ 2\langle \gamma^a\gamma^b\chi_a\partial_{x^b}\varphi,\psi\rangle_{g_S^\vee\otimes\varphi^*n} + \frac{1}{2}\langle\chi_a, \gamma^b\gamma^a\chi_b\rangle_{g_S}\langle\psi,\psi\rangle_{g_S^\vee\otimes\varphi^*n} \\
		+ \frac16\varepsilon^{\alpha\beta}\varepsilon^{\gamma\delta}\langle R^{\varphi^*TN}(\psi_\alpha, \psi_\gamma)\psi_\delta, \psi_\beta\rangle_{\varphi^*n} \d vol_g
	\end{multline}
\end{thm}
Notice that the symmetries of \(A(M,\Phi)\) translate into several symmetries for \(A(\varphi, g, \psi, \chi, F)\).
The \(G\)-invariance of the action \(A(M, \Phi)\) leads to conformal and super Weyl invariance of \(A(\varphi, g, \psi, \chi, F)\).
The invariance of \(A(M, \Phi)\) under super diffeomorphisms splits into diffeomorphism invariance and super symmetry of \(A(\varphi, g, \psi, \chi, F)\).
Super symmetry of \(A(\varphi, g, \psi, \chi, F)\) is the invariance up to first order under an infinitesimal change of the embedding \(i\) parametrized by the spinor \(q\).
The formulas for the super symmetry between \(g\) and \(\chi\) have been given in Equation~\eqref{eq:MetricGravitinoSUSY}.
A full calculation for the super symmetry of \(\varphi\), \(\psi\) and \(F\) can be found in~\cite{EK-DR}.
We give here the resulting formulas for the special case \(F=0\) and \(R^N = 0\):
\begin{align}
	\delta\varphi &= \langle q, \psi \rangle & \delta\psi &= \left(\partial_{x^k}\varphi - \langle\psi, \chi_k\rangle\right)\gamma^k q
\end{align}
We have thus given a super geometric explanation to all symmetries of the action functional \(A(\varphi, g, \psi, \chi, F)\) appearing in the literature (e.g. \cites{BdVH-LSRIASS}{DZ-CASS}).

We now turn to applications of the action functional \(A(M, \Phi)\) to the super moduli space or super Teichmüller space.
Similar to the case of Riemann surfaces and the harmonic action functional, the action functional \(A(M, \Phi)\) can be seen as a functional on the super moduli space.
By the left hand side of Equation~\eqref{eq:SuperModuliQuotient} the super moduli space is given by the integrable \(G\)-structures up to super diffeomorphisms.
The action functional~\eqref{eq:SuperAF} depends explicitly on the \(G\)-structure and is super diffeomorphism invariant.
Thus one may expect that the action functional~\ref{eq:SuperAF} and its critical points---the maps \(\Phi\colon M\to N\) solving Equation~\eqref{eq:EL}---may be useful to study the super moduli space.
However certain difficulties arise from the presence of integrability conditions in Theorem~\ref{thm:SRSReductionOfStructureGroup}.
Let \(H\) be an infinitesimal variation of the \(G\)-frame \(F_A\), i.e.\ first derivative of a family of frames
\begin{equation}
	{F(t)}_A = F_A + tH_A^B F_B + o(t).
\end{equation}
If the family of frames \({F(t)}_A\) is a family of integrable \(G\)-frames, certain infinitesimal integrability conditions hold for \(H\).
Consequently, the \(H\) that do not fulfil those infinitesimal integrability conditions are not admissible infinitesimal deformations of the super Riemann surface defined by \(F_A\).

The variation of the action functional~\eqref{eq:SuperAF} with respect to the variation of the frames \(F_A\) can be written as
\begin{equation}
	\delta_{F_A} A(M, \Phi) = \int_{M} H\cdot T^{super} [\d vol]
\end{equation}
The tensor \(T^{super}\) can be seen as a super version of the energy-momentum-tensor in Equation~\eqref{eq:EnergyMomentumTensor}.
However it is not guaranteed that non-integrable infinitesimal deformations \(H\) lie in the kernel of the variation \(\delta_{F_A} A(M, \Phi)\).
Thus, in contrast to the case of Riemann surfaces, \(\delta_{F_A} A(M, \Phi)\) can not be interpreted as a cotangent vector to the moduli space of super Riemann surfaces.

In order to circumvent the problem of integrability conditions one can turn to a description of the moduli space of super Riemann surfaces in terms of metrics and gravitinos, i.e.\ to the right hand side of Equation~\eqref{eq:SuperModuliQuotient}.
It is an advantage of the description of super Riemann surfaces in terms of metrics and gravitinos that there are no integrability conditions, i.e.\ every triple \((g, S, \chi)\) forms a super Riemann surface.
Thus every deformation of the given metric and given gravitino is an admissible deformation of the super Riemann surface at hand.
Define the energy-momentum tensor \(T\) of \(A(\varphi, g, \psi, \chi, F)\) via
\begin{equation}
	\delta_g A(\varphi, g, \psi, \chi, F) = \int_{\Smooth{M}} \delta g \cdot T \d{vol_g}
\end{equation}
and the super current \(J\) by
\begin{equation}
	\delta_\chi A(\varphi, g, \psi, \chi, F) = \int_{\Smooth{M}} \delta \chi \cdot J \d{vol_g}
\end{equation}
If the fields \(\varphi\), \(\psi\), and \(F\) are critical points of \(A(\varphi, g, \psi, \chi, F)\), then \(T\) is the Noether current with respect to diffeomorphism invariance, whereas \(J\) is the Noether current with respect to super symmetry.
It can be shown that \(T\) and \(J\) are components of \(T^{super}\) similar to \(\varphi\), \(\psi\) and \(F\) being components of \(\Phi\).
Thus, once again, diffeomorphism invariance and super symmetry are very much the same thing from the viewpoint of super geometry.
With the help of the diffeomorphism invariance and super symmetry of \(A(\varphi, g, \psi, \chi, F)\) one can show that \(T\) is, once again, a holomorphic quadratic differential and \(J\) a holomorphic section of \(S^\vee\otimes_\bbC S^\vee\otimes_\bbC S^\vee\).
They are even, resp.\ odd tangent vectors to the moduli space of super Riemann surfaces.

One can hope that the study of critical points of the action functional \(A(\varphi, g, \psi, \chi, F)\) turns out as useful for the study of the moduli space of super Riemann surfaces as the study of harmonic maps is for Teichmüller theory.
 
\printbibliography

\end{document}